\documentclass[11pt, oneside, a4paper]{article}
\usepackage[utf8]{inputenc}
\usepackage[T1]{fontenc}
\usepackage{amsmath}
\usepackage{amsthm,enumerate} 
\usepackage{amssymb}
\usepackage{listings}
\usepackage{url}
\usepackage{hyperref}
\usepackage{xcolor}
\usepackage{color}
\usepackage{dsfont}
\usepackage{tikz-cd}
\usepackage{cleveref}

\allowdisplaybreaks[4]

\newtheorem{thm}{Theorem}[section]
\newtheorem{lemma}[thm]{Lemma}
\newtheorem{result}[thm]{Result}

\newtheorem{cor}[thm]{Corollary}
\newtheorem{prop}[thm]{Proposition}
\newtheorem{notation}[thm]{Notation}
\newtheorem{remark}[thm]{Remark}
\newtheorem{l+d}[thm]{Lemma und Definition}
\newtheorem{k+d}[thm]{Corollary und Definition}

\newtheorem{example}[thm]{Example}

\newtheorem{bem+de}[thm]{Defintion+Bemerkung}

\newcommand{\EKR}{\mathcal{F}}
\newcommand{\OO}{\mathcal{O}}
\newcommand{\Gauss}[2]{{\begin{bmatrix} #1 \\ #2 \end{bmatrix}_q}}
\DeclareMathOperator{\PG}{PG}

\title{Cocliques in the Kneser graph on $(n-1,n)$-flags of $\PG(2n,q)$ }

\author{Philipp Heering 
\footnote{Justus-Liebig-Universität, Mathematisches Institut, Arndtstraße 2, D-35392 Gießen, Germany, philipp.heering@math.uni-giessen.de } }

\date{2026}

\begin{document}

\maketitle

\begin{center}
\subsection*{Abstract}
\end{center}

In the finite projective space PG$(2n,q)$ we consider flags of type $(n-1,n)$, that is, pairs $(A,B)$ consisting of an $(n-1)$-space $A$ and an $n$-space $B$ that are incident. Two such flags $(A_1,B_1)$ and $(A_2,B_2)$ are opposite if $A_1\cap B_2=A_2\cap B_1=\emptyset$. 
Let $\Gamma_{2n}$ be the graph whose vertices are the flags of type $(n-1,n)$ of PG$(2n,q)$, with two vertices being adjacent if the corresponding flags are opposite.
Using the Erd\H{o}s-Matching theorem for vector spaces shown by Ihringer, we determine, for $q$ large enough, the largest cocliques of $\Gamma_{2n}$ and obtain a stability result.
This EKR-type theorem 
proves a conjecture of D'haeseleer, Metsch and Werner. 

\textbf{Keywords:}  Erd\H{o}s-Ko-Rado, generalized Kneser graph, coclique\\
 \textbf{MSC (2020):} 
 05C69, 
 51E20, 
 05C35 

\section{Introduction}

Consider the projective space $\PG(2n,q)$ of dimension $2n$. An \emph{$i$-space} of $\PG(2n,q)$ is a subspace of $\PG(2n,q)$ with projective dimension $i$, in particular a $0$-space is a point, a $1$-space is a line and so on. Now, consider an $(n-1)$-space $A$ and an $n$-space $B$ of $\PG(2n,q)$. If $A$ is contained in $B$, we call $(A,B)$ an \emph{$(n-1,n)$-flag} of $\PG(2n,q)$. Two such flags $(A_1,B_1)$ and $(A_2,B_2)$ are called \emph{opposite} if $A_1\cap B_2=A_2\cap B_1=\emptyset$. For $n\geq 2$ let $\Gamma_{2n}$ be the graph whose vertices are the $(n-1,n)$-flags of $\PG(2n,q)$, with two vertices being adjacent if the corresponding flags are opposite.
We want to study cocliques of $\Gamma_{2n}$.

This research falls under the umbrella of Erd\H{o}s-Ko-Rado (EKR) problems and similar problems have been studied \cite{pointhyperplaneflags, point-plane, k=3_Hilton_Milner, IMM18, Thechromaticnumberoftwo}. In fact, the problem we are concerned with is a direct generalization of previous work.
Cocliques of $\Gamma_{2n}$ were studied by  Blokhuis and Brouwer \cite{cocliquesonlineplane4d} for $n=2$ and by Metsch and Werner \cite{werner_metsch} for $n=3$.

Out of all the EKR problems for flags, the problem for chambers of spherical buildings has enjoyed the most attention in recent years.
In \cite{AlgebraicApproach, AlgebraicApproach2, klaus_jesse_philipp} the largest EKR sets for chambers of odd dimensional projective spaces were characterized for $q$ large enough. 
On the other hand, for the EKR problem on chambers of even dimensional projective spaces very few results are known.
 The case $\PG(2,q)$ is probably folklore. The case $\PG(4,q)$ was solved in \cite{heering_PG4q} for $q$ large. 
However, the proof of \cite{heering_PG4q} regarding the EKR problem for chambers of $\PG(4,q)$, relies on a description of the largest cocliques of $\Gamma_{2n}$ for $n=2$ \cite{cocliquesonlineplane4d}.
 In particular \cite{heering_PG4q} relies also on the existence of a stability version of the results of \cite{cocliquesonlineplane4d}, which can be found in \cite{DHAESELEER2022103474}. This suggests that studying the largest cocliques of $\Gamma_{2n}$ may be helpful when attacking the EKR problem for chambers of $\PG(2n,q)$.

Next up, we describe large cocliques for $\Gamma_{2n}$.
\begin{example} \label{E: main}
\begin{enumerate} [(a)]
    \item Let $H$ be a hyperplane of $\PG(2n,q)$ and let $X$ be either $(i)$ a point of $H$, or $(ii)$ an $(2n-2)$-space incident with $H$.
    Now, let $\EKR$ be the set of all $(n-1,n)$-flags such that for every $(A,B)\in \EKR$ we have $B\subseteq H$, or $A$ is incident with $X$ and $H$.
    \item Let $P$ be a point of $\PG(2n,q)$ and let $X$ be either $(i)$ a hyperplane incident with $P$, or $(ii)$ a line incident with $P$.
    Now, let $\EKR$ be the set of all $(n-1,n)$-flags such that for every $(A,B)\in \EKR$ we have $P\subseteq A$, or $B$ is incident with $X$ and $P$.
\end{enumerate}
\end{example}

It was already observed in Theorem 5.1 of \cite{cocliquesonlineplane4d} that the construction of Example \ref{E: main} does not contain opposite flags, i.e. the construction yields a coclique of $\Gamma_{2n}$.
Observe that parts (a) and (b) of \Cref{E: main} are dual to one another.
There is also some form of duality between the parts $(i)$ and $(ii)$. Consider for example case $(a)$, then $X$ is either a point or a hyperplane of $H$.\\
Next up, we want to count how many flags are in a set $\EKR$ as described in \Cref{E: main} as a warmup.
Beforehand, we recall here the definition of the Gaussian coefficient. 
 \begin{align*}
 \Gauss{b}{a}=
 \begin{cases}
\displaystyle \prod_{i=1}^a\frac{q^{b-a+i}-1}{q^i-1} & \text{if $0\le a\le b$,}
\\
0 & \text{otherwise.} \end{cases}
\end{align*}
We remark that we expect the reader to be moderately familiar with combinatorics in projective spaces. The term $\Gauss{b}{a}$ is a $q$-polynomial of degree $a(b-a)$, so in particular $\Gauss{b}{a}$ is $\OO(q^{a(b-a)})$.\\
Now, we count the size of the set $\EKR$ as described in \Cref{E: main} $(a)(i)$ and remark that duality implies that all constructions described in \Cref{E: main} have the same size.

The number of $n$-spaces incident with a hyperplane $H$ of $\PG(2n,q)$ is $\Gauss{2n}{n+1}$. Each such $n$-space contains exactly $\Gauss{n+1}{1}$ many $(n-1)$-spaces. The number of $(n-1)$-spaces incident with the point $X$ and the hyperplane $H$ is $\Gauss{2n-1}{n-1}$. Each such $(n-1)$-space is incident with exactly $q^n$ many $n$-spaces that are not contained in $H$. Thus the set in question has size $\Gauss{2n}{n+1} \Gauss{n+1}{1} +\Gauss{2n-1}{n-1}q^n$. \\

In order to state our main result we need one more definition.
\begin{align*}
f(n,q):=\left\{\begin{array}{ll}
3\Gauss{n}{1}\Gauss{2n-2}{n-2} & \text{if $n\geq 4$ and $q\geq 4$,}
\\
q^5+2q^4+3q^3+2q^2+q+1 &  \text{if $n=3$,}
\end{array}
\right.
\end{align*}

\begin{thm} \label{T: main}
For $n\geq 3$ let $\Gamma_{2n}$ be the graph whose vertices are the $(n-1,n)$-flags of $\PG(2n,q)$, with two vertices being adjacent if the corresponding flags are opposite. For $q$ large enough compared to $n$, a maximal coclique $\EKR$ of $\Gamma_{2n}$ falls into exactly one of the following categories.
\begin{enumerate} [(a)]
    \item $|\EKR|=\Gauss{2n}{n-1} \Gauss{n+1}{1} +\Gauss{2n-1}{n-1}q^n$ and $\EKR$ is one of the structures described in Example \ref{E: main}.
    \item $|\EKR|\leq \Gauss{2n}{n-1} \Gauss{n+1}{1} +f(n,q)\cdot q^n$ and there is either a hyperplane $H$, such that $\EKR$ contains all flags $(A,B)$ with $B\subseteq H$, or there is a point $P$ such that $\EKR$ contains all flags $(A,B)$ with $P\subseteq A$.
    \item $|\EKR|$ is at most $\OO(q^{n^2+n-2})$.
\end{enumerate}
\end{thm}

It is quickly checked that $f(n,q)<\Gauss{2n-1}{n-1}$ for large $q$, as $\Gauss{2n-1}{n-1}$ is a $q$-polynomial of degree $(n-1)n$, whereas $f(n,q)$ is a $q$-polynomial of degree $(n-1)+(n-2)n=(n-1)n-1$.

For $n=3$ the theorem follows from \cite{werner_metsch}. 
We remark that it is possible to derive a more explicit bound in case (c) from our proof.
As this would add only marginal insight but substantially complicate the paper, we omit this.
It is noteworthy that our proof uses the EKR and Hilton-Milner Theorems for vector spaces.
The key ingredient, however, is the Erd\H{o}s-Matching Theorem for vector spaces, which was proven by Ihringer in 2021.

In 2023 the chromatic number of $\Gamma_{2n}$ was determined modulo a conjecture by D'haeseleer, Metsch and Werner \cite{Jozefien_some}. Theorem \ref{T: main} is precisely the confirmation that their conjecture is true. Therefore, by proving Theorem \ref{T: main} we immediately get the following.

\begin{cor}
    For $q$ large enough, the chromatic number of $\Gamma_{2n}$ is $\frac{q^{n+2}-1}{q-1}-q$.
\end{cor}

\section{Two colors}

Throughout this paper, we will denote by $\EKR$ a maximal coclique of $\Gamma_{2n}$.
 Here we draw the readers attention to the difference of maximal and maximum. To avoid confusion, we use the term largest coclique instead of maximum coclique.

\begin{lemma} \label{L: possible weights}
Let $\EKR$ be a maximal coclique of $\Gamma_{2n}$.
    Let $B$ be an $n$-space of $\PG(2n,q)$ that occurs in a flag of $\EKR$. There exists a $k\in\{1,\ldots,n+1\}$, such that the number of flags in $\EKR$ that contain $B$ is precisely $\Gauss{k}{1}$. Furthermore, the number in question is $\Gauss{n+1}{1}$ if and only if every flag $(A',B')\in \EKR$ satisfies $A'\cap B\neq \emptyset$.
\end{lemma}

\begin{proof}
First, assume that the number  of flags in $\EKR$ that contain $B$ is $\Gauss{n+1}{1}$. This means that for every $(n-1)$-space $A$ contained in $B$, we have $(A,B)\in \EKR$. Let $(A',B')$ be an $(n-1,n)$-flag of $\PG(2n,q)$ with $A'\cap B=\emptyset$. Then $B\cap B'$ is just a point $P$. There exist $(n-1)$-spaces in $B$ that are skew to $P$. Therefore, there are flags in $\EKR$ that are opposite to $(A',B')$, hence $(A',B')$ is not in $\EKR$.

    Now, assume that every flag $(A',B')\in \EKR$ satisfies $A'\cap B\neq \emptyset$. Then every flag $(A,B)$, where $A$ is an $(n-1)$-space contained in $B$, is non-opposite to all flags in $\EKR$. By maximality of $\EKR$, all such flags are in $\EKR$. The number of flags in $\EKR$ containing $B$ is $\Gauss{n+1}{n}=\Gauss{n+1}{1}$.

    Finally, assume that there exist $(n-1,n)$-flags in $\EKR$ whose $(n-1)$-spaces are skew to $B$. Let $(A'_1,B'_1),\ldots,(A'_k,B'_k)$ be all these flags. Now, any flag $(A,B)\in \EKR$ has to satisfy $A\cap B'_1\neq\emptyset,\ldots, A\cap B'_k\neq\emptyset$. 
    Let $1\leq i\leq k$.
    As $A'_i\cap B=\emptyset$, we know that $B'_i\cap B$ is a point, say $P_i$.
    Now $A$ has to be incident with all points $P_i$. In other words, the span of the points $P_i$ has to be incident with $A$. Let $s$ be the dimension of the span $S$ of the points $P_i$.
    We have $0\leq s\leq n-1$, otherwise there could be no flag in $\EKR$ that contains $B$. 
    By maximality of $\EKR$, any $(n-1)$-space $A$ that is incident with $S$ and $B$ occurs in a flag $(A,B)$ of $\EKR$. There are $\Gauss{n-s}{(n-1)-s}=\Gauss{n-s}{1}$ such flags.
\end{proof}

Using duality, we immediately get the following.

\begin{cor}
    \label{C: possible weights}
Let $\EKR$ be a maximal coclique of $\Gamma_{2n}$.
    Let $A$ be an $(n-1)$-space of $\PG(2n,q)$ that occurs in a flag of $\EKR$. There exists a $k\in\{ 1,\ldots,n+1\}$ such that the number of flags in $\EKR$ that contain $A$ is $\Gauss{k}{1}$. Furthermore, the number in question is $\Gauss{n+1}{1}$ if and only if every flag $(A',B')\in \EKR$ satisfies $B'\cap A\neq \emptyset$.
\end{cor}

Adapting the notation of Blokhuis and Brouwer \cite{cocliquesonlineplane4d}, we introduce the following.

\begin{notation}
Let $\EKR$ be a maximal coclique of $\Gamma_{2n}$.
    An $n$-space $B$, or an $(n-1)$-space $A$ that occurs in $\Gauss{n+1}{1}$ flags of $\EKR$ is called $\EKR$-red or, if there is no risk of confusion, simply red. If the space in question occurs only in $\leq \Gauss{n}{1}$ flags of $\EKR$, but at least in $1$, we call it $\EKR$-yellow or, if there is no risk of confusion, simply yellow.
\end{notation}

A similar definition can be found in \cite{werner_metsch}, here red spaces are called saturated.

Two $(n-1)$-spaces of $\PG(2n,q)$ are in general position if they do not intersect. Dually, two $n$-spaces of $\PG(2n,q)$ are in general position if they span the entire space and hence intersect only in a point. Therefore, two $n$-spaces of $\PG(2n,q)$ that are not in general position intersect in at least a line.

\begin{lemma} 
Let $\EKR$ be a maximal coclique of $\Gamma_{2n}$.
Any two red $(n-1)$-spaces always meet.
Any two red $n$-spaces are not in general position.
\end{lemma}

\begin{proof}
We only show the statement regarding $n$-spaces, the statement for $(n-1)$-spaces follows by duality.

By Lemma \ref{L: possible weights} an $n$-space $B$ is red if and only if every flag $(A',B')\in \EKR$ satisfies $A'\cap B\neq \emptyset$. If $B_1$ and $B_2$ are two $n$-spaces in general position, i.e. they intersect only in a point, then there exist $(n-1)$-spaces $A_1, A_2$ with $A_1\subseteq B_1$ and $A_2\subseteq B_2$, such that  $A_1\cap B_2=A_2\cap B_1=\emptyset$. 
\end{proof}


Next up, we need the EKR and Hilton-Milner Theorem for vector spaces.

\begin{result} (\cite{hiltonmilner_vectorspaces, ekr_vectorspaces}) \label{R: EKR + HM vectorspaces}
Let $\mathcal{C}$ be a set of pairwise intersecting $(n-1)$-spaces of $\PG(2n,q)$, with $n\geq 3$ and $q\geq 3$.
Then either $|\mathcal{C}|\leq \Gauss{2n}{n-1}-q^{n(n-1)}\Gauss{n}{1}$, or there exists a point $P$, such that every $(n-1)$-space is incident with $P$.
\end{result}

\begin{remark} \label{Remark 1}
We remark that $\Gauss{2n}{n-1}-q^{n(n-1)}\Gauss{n}{1}$ is a polynomial in $q$ of degree at most $n^2-2$.
\end{remark}

\begin{cor} \label{C: number of red spaces}
Let $\EKR$ be a maximal coclique of $\Gamma_{2n}$ with $n\geq 3$.
Either the number of red $(n-1)$-spaces in $\PG(2n,q)$ is $\OO(q^{n^2-2})$, or there is a point $P$ and every $(n-1)$-space incident with $P$ is red.\\
Either the number of red $n$-spaces in $\PG(2n,q)$ is $\OO(q^{n^2-2})$, or there is a hyperplane $H$ and every $n$-space incident with $H$ is red.
\end{cor}

\begin{proof}
    Assume that the number of red $(n-1)$-spaces in $\PG(2n,q)$ is not $\OO(q^{n^2-2})$, then it is in particular $> \Gauss{2n}{n-1}-q^{n(n-1)}\Gauss{n}{1}$ for $n\geq 3$. \Cref{R: EKR + HM vectorspaces} and \Cref{Remark 1} imply that there exists a point $P$, such that every red $(n-1)$-space is incident with $P$. 

    By \Cref{C: possible weights}, for every flag $(A,B)\in \EKR$, we have that it's $n$-space meets all red $(n-1)$-spaces.
    Let $B$ be an $n$-space that intersects all red $(n-1)$-spaces with $P\notin B$. The number of lines that meet $P$ and $B$ is $\Gauss{n+1}{1}$, as this is the number of points in $B$. The number of $(n-1)$-subspaces incident with a line is $\Gauss{2n-1}{n-2}$. Therefore, there are at most $\Gauss{n+1}{1}\cdot \Gauss{2n-1}{n-2}$ red $(n-1)$-spaces in this case and this number is $\OO(q^n)\cdot \OO(q^{(n-2)(n+1)})=\OO(q^{n^2-2})$, a contradiction.

Hence, for all flags $(A,B)\in \EKR$, it's $n$-space meets $P$. By maximality of $\EKR$, every $(n-1)$-space incident with $P$ is red. For $n$-spaces the statement follows by duality.
\end{proof}

\section{Many red subspaces}

Let $\EKR$ be a maximal coclique of $\Gamma_{2n}$.
In this section we focus on the case in which the number of red $n$-spaces is not $\OO(q^{n^2-2})$. Our goal is to show that in this case $\EKR$ is as described in Example \ref{E: main}. The arguments clearly work in an analogue way for $(n-1)$-spaces, we leave dualization to the reader.

Before we can accomplish our goal, we need to collect some results. In particular, we need the EKR and Hilton-Milner Theorem for $(n-1)$-spaces of $\PG(2n-1,q)$. 

\begin{result} \label{R: EKR for n=2k}
  Let $\mathcal{C}$ be a set of pairwise intersecting $(n-1)$-spaces of $\PG(2n-1,q)$. 
  \begin{itemize}
      \item[(a)] (Theorem 5.3.1 in \cite{Newman_PhD_Thesis}) Assume that $\mathcal{C}$ is as large as possible. Then there exists a point or hyperplane in $\PG(2n-1,q)$ that is incident with all $(n-1)$-spaces. In particular, $|\mathcal{C}|=\Gauss{2n-1}{n-1}$.
      \item[(b)]Now assume that $\mathcal{C}$ is a set of pairwise intersecting $(n-1)$-spaces of $\PG(2n-1,q)$, such that that there is no point or hyperplane in $\PG(2n-1,q)$ that is incident with all $(n-1)$-spaces.
          \begin{itemize}
\item[(a)] (Theorem 6 in \cite{Ihringer_Hilton_Milner}) Let $q\geq 4$ and $n\geq 4$. Then $|\mathcal{C}|<3\Gauss{n}{1}\Gauss{2n-2}{n-2}$.
\item[(b)] (Theorem 6.1 in \cite{k=3_Hilton_Milner})
  Let $n=3$. Then $|\mathcal{C}|\leq q^5+2q^4+3q^3+2q^2+q+1$.
\end{itemize}
  \end{itemize}
\end{result}

We recall the following definition 
\begin{align*}
f(n,q):=\left\{\begin{array}{ll}
3\Gauss{n}{1}\Gauss{2n-2}{n-2} & \text{if $n\geq 4$ and $q\geq 4$,}
\\
q^5+2q^4+3q^3+2q^2+q+1 &  \text{if $n=3$,}
\end{array}
\right.
\end{align*}

\begin{lemma}
    Let $\EKR$ be a maximal coclique of $\Gamma_{2n}$ with $n\geq 3$ in which the number of red $n$-spaces is not $\OO(q^{n^2-2})$. Then $\EKR$ is as in Example \ref{E: main} (a), or $|\EKR|\leq  \Gauss{2n}{n-1} \Gauss{n+1}{1} +f(n,q)\cdot q^n $ and there is a hyperplane $H$, such that $\EKR$ contains all flags $(A,B)$ with $B\subseteq H$.
\end{lemma}

\begin{proof}
    By \Cref{C: number of red spaces}, we know that there is a hyperplane $H$ such that every $n$-space incident with $H$ is red. By \Cref{L: possible weights}, we have for any flag $(A,B)$ in $\EKR$ that $A$ intersects all the $n$-spaces in $H$. This is only possible if $A\subseteq H$. The number of flags $(A,B)$ with $A,B\subseteq H$ is $\Gauss{2n}{n+1} \Gauss{n+1}{n}=\Gauss{2n}{n-1} \Gauss{n+1}{1}$.

    Let $(A_1,B_1)$ and $(A_2,B_2)$ be flags of $\EKR$ with $B_1,B_2$ not red, i.e. $B_1,B_2$ are not contained in $H$. If $A_1$ and $A_2$ are skew, then $B_1$ is skew to $A_2$, as $B_1$ intersects $H$ only in $A_1$. Similarly $B_2$ is skew to $A_1$. So if $A_1$ and $A_2$ are skew, then the flags $(A_1,B_1)$ and $(A_2,B_2)$ are opposite, which is a contradiction. 

    In conclusion the flags $(A,B)$ in $\EKR$ with $B\not\subseteq H$ have $(n-1)$-spaces that pairwise intersect in $H$. By \Cref{R: EKR for n=2k} there are two options. Either there is a point or $(2n-2)$-space in $H$ that is incident with $A$ for any such flag $(A,B)$, or there are at most $f(n,q)$ distinct $(n-1)$-subspaces $A$ in these flags. 
    If there is a point or $(2n-2)$-space in $H$ that is incident with $A$ for any flag $(A,B)$ with $B\not\subseteq H$, then $\EKR$ is as in Example \ref{E: main} (a).
    If no such point or $(2n-2)$-space exists, then we have $ \Gauss{2n}{n-1} \Gauss{n+1}{1}$ flags $(A,B)$ with $A,B\subseteq H$ and since any $(n-1)$-space in $H$ is incident with $q^n$ many $n$-spaces that are not contained in $H$, we have at most $q^n\cdot f(n,q)$ many flags $(A,B)$ with $B\not\subseteq H$. In conclusion, we have at most $\Gauss{2n}{n-1} \Gauss{n+1}{1} +f(n,q)\cdot q^n$ flags in this case.
\end{proof}


\section{Few red subspaces}

Let $\EKR$ be a maximal coclique of $\Gamma_{2n}$.
Our goal in this section is to show that $|\EKR|$ is $\OO(q^{n^2+n-2})$, if the number of red spaces is $\OO(q^{n^2-2})$.

First up, we want to show that we must have many yellow $(n-1)$-spaces that are pairwise skew. We need the Erd\H{o}s-Matching Theorem for vector spaces.

\begin{result} (\cite{Ihringer_EM}) \label{R: EM for vector spaces}
Let $n\geq 3$ and let $\mathcal{C}$ be a family of $(n-1)$-spaces in $\PG(2n,q)$, such that $\mathcal{C}$ does not contain $s+1$ pairwise skew $(n-1)$-spaces. Then $|\mathcal{C}|\leq s \Gauss{2n}{n-1}$, for $q$ large enough compared to $n$.
\end{result}

\begin{lemma} \label{L: many yellow pw skew}
     Let $\EKR$ be a maximal coclique of $\Gamma_{2n}$ with $n\geq 3$ and assume that the number of red spaces is $\OO(q^{n^2-2})$. Then $|\EKR|$ is $\OO(q^{n^2+n-2})$, or there are at least $n+1$ yellow $(n-1)$-spaces in $\EKR$ that are pairwise skew, or there are at least $n+1$ yellow $n$-spaces in general position.
\end{lemma}

\begin{proof}
    The number of red spaces is $\OO(q^{n^2-2})$ and a red space is contained in $\Gauss{n+1}{1}$ flags of $\EKR$, hence the number of flags in $\EKR$ that contain a red space is $\OO(q^{n^2-2})\cdot \Gauss{n+1}{1}=\OO(q^{n^2+n-2})$. 

    Let us assume that $|\EKR|$ is not $\OO(q^{n^2+n-2})$. 
    As a yellow space is contained in at most $\Gauss{n}{1}$ flags of $\EKR$, the number of yellow spaces cannot be $\OO(q^{n^2-1})$. If there are at most $n$ yellow $(n-1)$-spaces in $\EKR$ that are pairwise skew, \Cref{R: EM for vector spaces} implies that the number of yellow $(n-1)$-spaces is at most $n \Gauss{2n}{n-1}$ which is $\mathcal{O}(q^{n^2-1})$. For yellow $n$-spaces the result follows by duality.
\end{proof}

Now, we show that there are few $n$-spaces that intersect $n+1$ pairwise skew $(n-1)$-spaces.

\begin{lemma} \label{L: few that meet all}
 Let $A_1,\ldots,A_{n+1}$ be pairwise skew $(n-1)$-spaces of $\PG(2n,q)$. The number of $n$-spaces that intersect every $A_i$ for $1\leq i\leq n+1$ is $\OO(q^{n^2-1})$.
\end{lemma}

\begin{proof}
     Let $B$ be an $n$-space that meets $A_1,\ldots,A_{n+1}$. 
     Then $B\cap A_i$ has to be some subspace $P_i$ for all $1\leq i\leq n+1$.
First, we will count subspaces $P_i$ in $A_i$ (for $i$ in a subset of $\{1,\ldots,n+1\}$) and then we will consider all $n$-spaces incident with the span of the $P_i$. This gives an upper bound on the number of $n$-spaces that meet $A_1,\ldots,A_{n+1}$.
As a first step, we can choose dimensions $s_1,\ldots,s_{n+1}$, where $s_i=dim(P_i)$. As $0\leq s_i\leq n-1$, there are $n$ options for every $i$, so in total at most $n^{n+1}$ options. From now on, we assume that $s_1,\ldots,s_{n+1}$ are fixed.

We distinguish two cases: Either we choose $P_1,\ldots,P_{n+1}$ in such a way that there is an $i$, such that the span $\langle P_1,\ldots,P_{i'} \rangle$ meets $A_{i'+1}$, or we choose $P_1,\ldots,P_{n+1}$ such that there is no such $i'$.

\textbf{\underline{Case 1:}} We choose $P_1,\ldots,P_{n+1}$, such that there exists an $i'\in \{2,\ldots,n\}$, with the property that the span $\langle P_1,\ldots,P_{i'} \rangle$ meets $A_{i'+1}$. Let $i$ be the smallest integer with this property.\\

This implies that $\langle P_1,\ldots,P_{i-1} \rangle$ does not meet $A_i$, hence $\langle P_1,\ldots,P_{i-1} \rangle$ is skew to $P_i$. Furthermore, $\langle P_1,\ldots,P_{i-2} \rangle$ does not meet $A_{i-1}$, hence $\langle P_1,\ldots,P_{i-2} \rangle$ is skew to $P_{i-1}$, so $\langle P_1,\ldots,P_{i-2} \rangle$, $P_{i-1}$ and $P_i$ are pairwise skew. Using the same argument $i-1$ times yields that $P_1,\ldots,P_i$ are pairwise skew.

In particular, we have
\begin{align} \label{A: dim P_1,...,P_i}
\begin{split}
    dim(\langle P_1,\ldots,P_{i} \rangle)&=dim( P_1)+\ldots+dim(P_i)+(i-1)\\
    &= s_1+\ldots+s_i+(i-1)   
\end{split}
\end{align}
and
\begin{align} \label{A: dim P_1,...,P_i-1}
    dim(\langle P_1,\ldots,P_{i-1} \rangle)=s_1+\ldots+s_{i-1}+(i-2). 
\end{align}
For any $j\leq i-1$ the number of $s_j$-subspaces in $A_j$ is $\Gauss{n}{s_j+1}$.
 As $P_1,\ldots,P_{i-1}$ are pairwise skew, the number of distinct subspaces $\langle P_1,\ldots,P_{i-1} \rangle$ is bound by the number of tuples $( P_1,\ldots,P_{i-1})$. So, the number of subspaces $\langle P_1,\ldots,P_{i-1} \rangle$ is at most
 $\Gauss{n}{s_1+1} \ldots   \Gauss{n}{s_{i-1}+1}$.

For any subspace $P_i$ skew to $\langle P_1,\ldots,P_{i-1} \rangle$, we have that $\langle P_1,\ldots,P_{i-1},P_i \rangle=\langle \langle P_1,\ldots,P_{i-1} \rangle,P_i \rangle$ consists of all points that are on a line $g$ that intersects $P_i$ and $\langle P_1,\ldots,P_{i-1} \rangle$.\\
 Let $Q_1$ be a point in  $\langle P_1,\ldots,P_{i-1} \rangle$. As $Q_1$ is in  $\langle P_1,\ldots,P_{i-1} \rangle$, which is skew to the $(n-1)$-space $A_i$, the span $S$ of $Q_1$ and $A_i$ has dimension $n$. If we assume that $S$ meets $A_{i+1}$ in at least a line, then $A_i$ (which has codimension $1$ in $S$) would meet $A_{i+1}$, which is a contradiction. 
 So, $S$ meets $A_{i+1}$ in at most a point $Q_3$. If such a point $Q_3$ exists, then the line $Q_1Q_3$ intersects $P_i$ in a point.
In other words: For every point $Q_1$ in $\langle P_1,\ldots,P_{i-1} \rangle$ there is at most one point $Q_2$ in $A_i$ with the property that $Q_1Q_2$ intersects $A_{i+1}$.
 Recall from (\ref{A: dim P_1,...,P_i-1}) that $\langle P_1,\ldots,P_{i-1}\rangle$ has dimension $s_1+\ldots+s_{i-1}+(i-2)$.
 Therefore, $\langle P_1,\ldots,P_{i-1}\rangle$ contains exactly $\Gauss{s_1+\ldots+s_{i-1}+ (i-1)}{1}$ many points $Q_1$.
Hence, the number of points $Q_2$ in $A_i$, with the property that there is a $Q_1$ in $\langle P_1,\ldots,P_{i-1}\rangle$ such that $Q_1Q_2$ intersects $A_{i+1}$, is at most  $\Gauss{s_1+\ldots+s_{i-1}+ (i-1)}{1}$. \\
As $\langle P_1,\ldots,P_{i} \rangle$ meets $A_{i+1}$, we have that $P_i$ is an $s_i$-subspaces in $A_i$ that is incident with one such point $Q_2$. The number of $s_i$-subspaces in $A_i$ incident with a fixed point is $\Gauss{n-1}{s_i}$.
 So, if $P_1,\ldots,P_{i-1}$ are fixed, there are at most $\Gauss{s_1+\ldots+s_{i-1}+ (i-1)}{1}\Gauss{n-1}{s_i}$ options for $P_i$. \\ 
Recall from (\ref{A: dim P_1,...,P_i}) that $\langle P_1,\ldots,P_{i} \rangle$ has dimension $(s_1+\ldots+s_i+i-1)$.
Therefore, the number of $n$-spaces $B$ that contain  $\langle P_1,\ldots,P_{i} \rangle$ is at most
 $\Gauss{2n-(s_1+\ldots+s_i+i-1)}{n-(s_1+\ldots+s_i+i-1)}$.  
Any $n$-space that contains $\langle P_1,\ldots,P_{n+1} \rangle$, must also contain $\langle P_1,\ldots,P_{i} \rangle$. Hence, we have at most
    \begin{align*}
        &\Gauss{n}{s_1+1} \ldots   \Gauss{n}{s_{i-1}+1}  \cdot \Gauss{s_1+\ldots+s_{i-1}+i-1}{1}\Gauss{n-1}{s_i}   \cdot\Gauss{2n-(s_1+\ldots+s_i+i-1)}{n-(s_1+\ldots+s_i+i-1)}
    \end{align*}
   many $n$-subspaces $B$ in Case 1. This number is a $q$-polynomial of degree at most
\begin{align*}
    &\left( \sum\limits_{j=1}^{i-1}(s_j+1)(n-s_j-1) \right)
    +(i-2+\sum\limits_{j=1}^{i-1}s_j)+s_i(n-1-s_i)+(n-i+1-\sum\limits_{j=1}^{i}s_j)n\\
    &=n^2 -\sum\limits_{j=1}^{i}(s_j^2+s_j) - 1\leq n^2-1.
\end{align*}

This implies that in this case the number in question is $\OO(q^{n^2-1})$.\\

\textbf{\underline{Case 2:}} We choose $P_1,\ldots,P_{n+1}$, such that there exists  no $i'\in \{2,\ldots,n\}$, with the property that the span $\langle P_1,\ldots,P_{i'} \rangle$ meets $A_{i'+1}$. \\

We start with the same argument as in Case 1. 
The span $\langle P_1,\ldots,P_{n} \rangle$ does not meet $A_{n+1}$, hence $\langle P_1,\ldots,P_{n} \rangle$ is skew to $P_{n+1}$. 
Furthermore, $\langle P_1,\ldots,P_{n-1} \rangle$ does not meet $A_{n}$, hence $\langle P_1,\ldots,P_{n-1} \rangle$ is skew to $P_{n}$, so $\langle P_1,\ldots,P_{n-1} \rangle$, $P_{n}$ and $P_{n+1}$ are pairwise skew. 
Using the same argument $n$ times yields that $P_1,\ldots,P_{n+1}$ are pairwise skew. Hence, 
$\langle P_1,\ldots,P_{n+1}\rangle$ has dimension $s_1+\ldots+s_{n+1}+n$.\\
Let $B$ be an $n$-space that contains $P_1,\ldots,P_{n+1}$.
As $B$ has dimension $n$, we have $s_1+\ldots+s_{n+1}+n\leq n$, hence $s_i=0$ for all $1\leq i\leq n+1$ and every subspace $P_i$ is a point. In particular, $B$ is the span of all points $P_i$. There are $\Gauss{n}{1}$ points in every $(n-1)$-space $A_i$, hence there are at most $\Gauss{n}{1}^{n+1}$ many $n$-spaces $B$ that are equal to $\langle P_1,\ldots,P_{n+1} \rangle$ and this number is $\mathcal{O}(q^{n^2-1})$.\\

In conclusion we see that in both cases the number in question is $\OO(q^{n^2-1})$.
\end{proof}

Next up, we show that a singe yellow $n$-space must meet almost all $(n-1)$-spaces of flags in $\EKR$, if $\EKR$ is large.

\begin{lemma} \label{L: yellows meet almost all}
    Let $\EKR$ be a maximal coclique of $\Gamma_{2n}$ with $n\geq 3$ and let $B$ be a yellow $n$-space of $\PG(2n,q)$. The number of flags $(A',B')\in \EKR$ with pairwise distinct $n$-spaces and with $A'\cap B=\emptyset$ is $\OO(q^{n^2-1})$.
\end{lemma}

\begin{proof}
    Since $B$ is a yellow $n$-space there exists an $(n-1)$-space $A$, such that $(A,B)\in \EKR$. Therefore, any flag $(A',B')\in \EKR$ with $A'\cap B=\emptyset$ has to satisfy that the point $B'\cap B$ is in  $A$. 

    Consider two flags $(A'_1,B'_1),(A'_2,B'_2)\in \EKR$ with $A'_1\cap B=\emptyset$ and $A'_2\cap B=\emptyset$ and assume $B'_1\cap B=B'_2\cap B=P$.
    If $B'_1\cap B'_2=P$, then $A'_1\cap B'_2=\emptyset=A'_2\cap B'_1$ as $A'_1$ and $A'_2$ are skew to $B$. Then the flags $(A'_1,B'_1),(A'_2,B'_2)$ are opposite, which is a contradiction. Therefore, the $n$-spaces $B'_1$ and $B'_2$ have to meet in at least a line containing $P$.
    In the quotient space of $P$, which has dimension $2n-1$, this means that the $(n-1)$-spaces corresponding to $B'_1$ and $B'_2$ have to intersect. By \Cref{R: EKR for n=2k}, there are at most $\Gauss{2n-1}{n-1}$ such spaces.

    In conclusion, there are $\Gauss{n}{1}$ points in $A$ and for every point $P$ in $A$ there are at most $\Gauss{2n-1}{n-1}$ many flags $(A',B')\in \EKR$ with distinct $n$-spaces, such that $A'\cap B=\emptyset$ and $B\cap B'=P$. Since $\Gauss{n}{1}\cdot \Gauss{2n-1}{n-1}$ is $\OO(q^{n-1})\cdot \OO(q^{(n-1)n})=\OO(q^{n^2-1})$ the statement follows.
\end{proof}

\begin{prop}
Let $n\geq 3$ and let $\EKR$ be a maximal coclique of $\Gamma_{2n}$ such that the number of red spaces in $\EKR$ is $\OO(q^{n^2-2})$. Then $|\EKR|$ is $\OO(q^{n^2+n-2})$.
\end{prop}

\begin{proof}
 The number of red spaces is $\OO(q^{n^2-2})$ and a red space is contained in $\Gauss{n+1}{1}$ flags of $\EKR$, hence the number of flags in $\EKR$ that contain a red space is $\OO(q^{n^2-2})\cdot \Gauss{n+1}{1}=\OO(q^{n^2+n-2})$. 
   We show that the number of flags in $\EKR$ with a yellow space is also $\OO(q^{n^2+n-2})$. Assume that this is not the case.

     By \Cref{L: many yellow pw skew}, there are $n+1$ pairwise skew yellow $(n-1)$-spaces, or $n+1$ yellow $n$-spaces in pairwise general position in $\PG(2n,q)$. We assume that $\PG(2n,q)$ contains $n+1$ pairwise skew $(n-1)$-spaces $A_1,\ldots,A_{n+1}$ and show that the number of flags in $\EKR$ is $\OO(q^{n^2+n-2})$, for yellow $n$-spaces dual arguments suffice.
     
     Let $1\leq i\leq n+1$. 
     Any flag $(A,B)\in \EKR$ satisfies $B\cap A_i=\emptyset$ or $B\cap A_i\neq\emptyset$.
     By the dual statement of \Cref{L: yellows meet almost all}, the number of flags $(A,B)\in \EKR$ with distinct $(n-1)$-spaces and with $A_i\cap B=\emptyset$ is at most $\OO(q^{n^2-1})$. As the red spaces are already accounted for, we can safely assume that these $(n-1)$-spaces are all yellow. This yields that the number of flags $(A,B)\in \EKR$ with $A_i\cap B=\emptyset$ is $\OO(q^{n^2-1})\Gauss{n}{1}=\OO(q^{n^2+n-2})$. 
     
     Hence, the number of flags $(A,B)\in \EKR$ with $A_i\cap B=\emptyset$ for some $1\leq i\leq n+1$ is $(n+1)\OO(q^{n^2+n-2})=\OO(q^{n^2+n-2})$. All other flags in $\EKR$ must have an $n$-space that meets $A_1,\ldots,A_{n+1}$. By \Cref{L: few that meet all}, the number of $n$-spaces that meet $A_1,\ldots,A_{n+1}$ is $\OO(q^{n^2-1})$. As the red spaces are already accounted for, we can safely assume that these $n$-spaces are all yellow. This yields that the number of flags $(A,B)\in \EKR$ with $A_i\cap B\neq\emptyset$ for all $1\leq i\leq n+1$  is also $\OO(q^{n^2-1})\Gauss{n}{1}=\OO(q^{n^2+n-2})$. 
\end{proof}

\subsection*{Acknowledgement}

The author would like to thank Klaus Metsch for many helpful discussions.

\bibliographystyle{plainurl}
\bibliography{bib}

\end{document}